\documentclass[a4paper,11pt,twoside]{amsart}

\usepackage{amsmath}
\usepackage{amsthm}
\usepackage{amsfonts}
\usepackage{amssymb}
\usepackage{mathrsfs}
\usepackage{graphicx}
\usepackage{enumerate}
\usepackage{url}

\usepackage{wasysym}

\usepackage{tikz}
\usepackage{tikz-cd}

\usetikzlibrary{decorations.pathreplacing}

\usepackage{subcaption}

\usetikzlibrary{arrows}
\usetikzlibrary{decorations.markings}

\tikzstyle{CCC}=[shape=circle, draw, fill=black!10, align=center, font=\scriptsize]
\tikzstyle{CC}=[shape=circle, draw, align=center, font=\scriptsize]

\theoremstyle{plain}
\newtheorem{lemma}{Lemma}[section]
\newtheorem{prop}[lemma]{Proposition}

\newtheorem{coro}[lemma]{Corollary}
\theoremstyle{remark}
\newtheorem{rem}[lemma]{Remark}
\newtheorem*{notat}{Notation}
\theoremstyle{definition}
\newtheorem{definition}[lemma]{Definition}
\newtheorem{ex}[lemma]{Example}

\newcommand{\N}{\mathbb{N}}

\newcommand{\Cc}{\mathscr{C}}
\newcommand{\Uc}{\mathscr{U}}
\newcommand{\Dc}{\mathscr{D}}
\newcommand{\Gm}{\Gamma}

\newcommand{\op}{\textup{op}}

\newcommand{\mxl}{\textup{mxl}}
\newcommand{\mnl}{\textup{mnl}}

\newcommand{\st}[1]{\textnormal{st}(#1)}

\newcommand{\I}{\mathcal{I}}

\newcommand{\modu}[1]{\ \ \textnormal{mod}\ #1}

\begin{document}

\title{The crosscut poset and the fixed point property}

\author{Ana Gargantini}
\address{Facultad de Ciencias Exactas y Naturales \\ Universidad Nacional de Cuyo \\ Mendoza, Argentina.}
\email{anagargantini@gmail.com}

\author{Miguel Ottina}
\address{}
\email{miguelottina@gmail.com}




\thanks{Research partially supported by grants M044 (2016--2018) and 06/M118 (2019--2021) of Universidad Nacional de Cuyo.}

\begin{abstract}
We introduce two novel ideas related to the crosscut poset and give many examples of application of these ideas to the fixed point property.
\end{abstract}

\maketitle

\section{Introduction}

The crosscut poset is a combinatorial invariant which is finer than the well-known crosscut complex. It was introduced by M. Ottina in \cite{ottina2022crosscut} and used to give generalizations of several important results such as Bj\"orner's crosscut Theorem.

In this article we show how an order-preserving map $f\colon P\to Q$ between posets induces an order-preserving map between certain crosscut posets associated to $P$ and $Q$. Moreover, an additional novel idea that we present in this article is that, given a poset $P$, the crosscut posets of $P$ with respect to the cutsets of maximal and minimal elements of $P$ can be combined to obtain a new poset, which we denote $\Cc(P)$, that has the additional property that any order-preserving map $f\colon P\to Q$ between posets induces an order-preserving map $\Cc(f)\colon \Cc(P)\to \Cc(Q)$, which restricts to order-preserving maps between the corresponding crosscut posets that make up $\Cc(P)$ and $\Cc(Q)$. This assigment, which is studied in the main section of this article, turns out to be very useful to study the fixed point property, as it is shown by several results that we present in that section and many examples that we study in the last section. In particular, we obtain an alternative proof to a result of \cite{hoft1976some} and we prove that if $P$ is a poset such that the crosscut poset of $P$ with respect to the maximal elements of $P$ has the fixed point property and every element of this crosscut poset has the fixed point property then the poset $P$ has the fixed point property, generalizing a result of \cite{baclawski1979fixed}.

\section{Preliminaries}

{\bf Irreducible points.} Let $P$ be a poset and let $x\in P$. We say that $x$ is an \emph{irreducible point} of $P$ if either $P_{< x}$ has a maximum element or $P_{>x}$ has a minimum element. Recall that if $x$ is an irreducible point of a poset $P$ then $P$ has the fixed point property if and only if $P-\{x\}$ has the fixed point property \cite[Scholium 4.13]{schroder2016ordered}.

We say that a finite poset $P$ is \emph{dismantlable by irreducibles} if there exists $n\in \N$ and a finite sequence $P_0,P_1,\ldots,P_n$ of subposets of $P$ such that $P_0=P$, the subposet $P_n$ has only one element and for each $j\in\{1,2,\ldots,n\}$, the subposet $P_j$ is obtained from $P_{j-1}$ by removing an irreducible point of $P_{j-1}$. It follows that if a finite poset is dismantlable by irreducibles then it has the fixed point property (\cite[Corollary 2]{rival1976fixed}). 

\begin{rem} \label{rem_crown}
Let $n\in\N$ such that $n\geq 2$ and let $P$ be a $2n$--crown, that is, a poset which is isomorphic to the poset defined by the following Hasse diagram
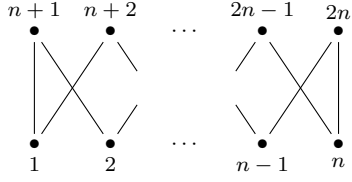
\begin{center}
\begin{tikzpicture}[y=1.5cm]
\tikzstyle{every node}=[font=\scriptsize]
\foreach \x in {1,2} \draw (\x,0) node(c\x)[inner sep=2pt]{$\bullet$} node[below=1]{$\x$}; 
\foreach \x in {3} \draw (\x,0) node(c\x)[inner sep=2pt]{$\cdots$};
\foreach \x in {4} \draw (\x,0) node(c\x)[inner sep=2pt]{$\bullet$} node[below=1]{${n-1}$}; 
\foreach \x in {5} \draw (\x,0) node(c\x)[inner sep=2pt]{$\bullet$} node[below=1]{$n$}; 

\foreach \x in {1,2} \draw (\x,1) node(b\x)[inner sep=2pt]{$\bullet$} node[above=1]{$n+\x$}; 
\foreach \x in {3} \draw (\x,1) node(b\x)[inner sep=2pt]{$\cdots$};
\foreach \x in {4} \draw (\x,1) node(b\x)[inner sep=2pt]{$\bullet$} node[above=1]{$2n-1$}; 
\foreach \x in {5} \draw (\x,1) node(b\x)[inner sep=2pt]{$\bullet$} node[above=1]{$2n$}; 

\foreach \x in {1,2} \draw (b1)--(c\x);
\foreach \x in {1} \draw (b2)--(c\x);
\draw[shorten >=1cm,shorten <=0cm] (b2)--(c3);
\draw[shorten >=1cm,shorten <=0cm] (c2)--(b3);
\draw[shorten >=1cm,shorten <=0cm] (b4)--(c3);
\draw[shorten >=1cm,shorten <=0cm] (c4)--(b3);
\foreach \x in {5} \draw (b4)--(c\x);
\foreach \x in {4,5} \draw (b5)--(c\x);
\end{tikzpicture}
\end{center}
Note that, for each element $a\in P$, the poset $P-\{a\}$ is dismantlable by irreducibles. It follows that every order-preserving map $f\colon P\to P$ which is not bijective has a fixed point.
\end{rem}

{\bf Crosscut poset.} We recall below the definition of the crosscut poset from \cite{ottina2022crosscut}. First, we introduce notation.

\begin{notat}
Let $S$ be a set. The power set of $S$ will be denoted by $\mathcal{P}(S)$ and the set $\mathcal{P}(S)-\{\varnothing \}$ will be denoted by $\mathcal{P}_{\neq\varnothing}(S)$. In addition, the set whose elements are the finite non-empty subsets of $S$ will be denoted by $\mathcal{P}_{f,\neq \varnothing}(S)$.
\end{notat}

\begin{definition}[\cite{ottina2022crosscut}]
Let $P$ be a poset and let $X$ be a subset of $P$. The \emph{crosscut poset of $P$ with respect to $X$} is the subposet of $(\mathcal{P}_{\neq\varnothing}(P),\subseteq)$ whose elements are the connected components of the non-empty subposets $\st{A}$ with $A\in \mathcal{P}_{\neq\varnothing}(X)$. It will be denoted by $\Gm(P,X)$.
\end{definition}

We also recall the following Proposition from \cite{ottina2022crosscut} that will be needed later.

\begin{prop}[\cite{ottina2022crosscut}] \label{prop_connected_subset}
Let $P$ be a poset and let $X$ be a subset of $P$. Let $B\subseteq P$ be a connected non-empty subposet and let $\Gamma_B=\{C\in \Gm(P,X) \mid B\subseteq C \}$. If $\Gamma_B$ is non-empty, then $\I_X(B)\neq \varnothing$ and $\Gamma_B$ has a minimum element, which is the connected component of $\st{\I_X(B)}$ that contains $B$.

In particular, if $B\in\Gm(P,X)$ then $\I_X(B)\neq \varnothing$ and $B$ is a connected component of $\st{\I_X(B)}$.
\end{prop}

\section{New constructions and results}

In this section we will show that, when the considered cutsets are that of maximal (or minimal) elements of the given posets, any order-preserving map between posets induces an order-preserving map between the corresponding crosscut posets. Besides, we will illustrate how this assigment can be used to study the fixed point property.

\begin{definition}
Let $P$ be a poset. 
\begin{itemize}
\item We define $\Dc(P)$ as the poset $\Gm(P,\mxl(P))$.
\item We define $\Uc(P)$ as the poset $\Gm(P,\mnl(P))^{\op}$.
\item  We define $\Cc(P)$ as the union of the posets $\Dc(P)$ and $\Uc(P)$ with the additional relations $C_1\leq C_2$ for every $(C_1,C_2)\in \Uc(P) \times \Dc(P)$ such that $C_1\cap C_2 \neq \varnothing$.
\end{itemize}
\end{definition}

Let $P$ be a poset. In the previous section we observed that the elements of $\Dc(P)$ are down-sets of $P$. Dually, the elements of $\Uc(P)$ are up-sets of $P$. Also, it is not difficult to check that if $P$ is a poset then $\Dc(P^{\op}) =\left(\Uc(P) \right)^{\op}$ and $\Cc(P^{\op}) =\left(\Cc(P) \right)^{\op}$. Thus, the results of this section concerning the poset $\Dc(P)$ admit dual versions involving the poset $\Uc(P)$.

\begin{ex} \label{ex_easy}
In this simple example we consider the poset $P$ of Example \ref{ex_2_sect_3} and construct its associated poset $\Cc(P)$.
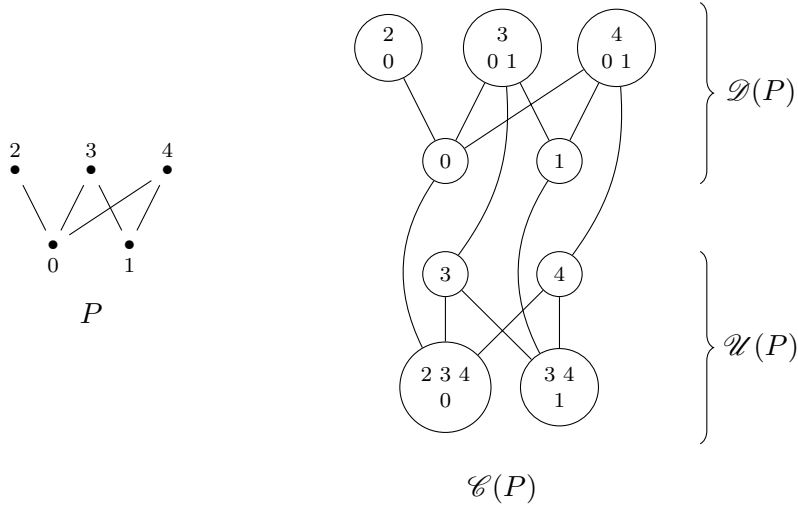
\medskip

\begin{center}
\begin{tikzpicture}[baseline=-3.55cm]
\tikzstyle{every node}=[font=\scriptsize]
\draw (0.5,-0.9) node(P){\normalsize $P$};
\foreach \x in {0,1} \draw (\x,0) node(\x){$\bullet$} node[below=1]{\x}; 
\foreach \x in {2,3,4} \draw (\x-2.5,1) node(\x){$\bullet$} node[above=1]{\x}; 
\foreach \x in {2,3,4} \draw (0)--(\x);
\foreach \x in {3,4} \draw (1)--(\x);
\end{tikzpicture}
\hspace*{2cm}
\begin{tikzpicture}[x=1.5cm, y=1.5cm]
\draw (1,-0.9) node(CP){\normalsize $\Cc(P)$};
\path node (U2) at (0,3) [CC] {2\\0}
node (U3) at (1,3) [CC] {3\\0 1}
node (U4) at (2,3) [CC] {4\\0 1}
 
node (U34a) at (0.5,2) [CC] {0} 
node (U34b) at (1.5,2) [CC] {1} 
 
node (F01a) at (0.5,1) [CC] {3}
node (F01b) at (1.5,1) [CC] {4}

node (F0) at (0.5,0) [CC] {2 3 4\\0}
node (F1) at (1.5,0) [CC] {3 4\\1};

\draw [-] (U34a) to (U3) to (U34b) to (U4) to (U34a) to (U2);
\draw [-] (F0) to (F01a) to (F1) to (F01b) to (F0);
\draw[-,bend left=30] (F0) to (U34a);
\draw[-,bend left=30] (F1) to (U34b);
\draw[-,bend right=20] (F01a) to (U3);
\draw[-,bend right=20] (F01b) to (U4);

\draw [decorate,decoration={brace,amplitude=5pt,mirror},xshift=-0.4cm]
(3,1.8) -- (3,3.4) node [black,midway,xshift=0.8cm] {\normalsize $\Dc(P)$};

\draw [decorate,decoration={brace,amplitude=5pt,mirror},xshift=-0.4cm]
(3,-0.5) -- (3,1.2) node [black,midway,xshift=0.8cm] {\normalsize $\Uc(P)$};

\end{tikzpicture}
\end{center}
\end{ex}

The following proposition shows that, for a connected poset $P$, the sets $\Dc(P)$ and $\Uc(P)$ of the previous definition are disjoint, except in a trivial case. It will be needed for Definition \ref{def_Df_Uf_Cf}.

\begin{prop}
Let $P$ be a connected poset. The following are equivalent:
\begin{enumerate}[(a)]
\item $\Dc(P)\cap \Uc(P)\neq\varnothing$.
\item $P$ has a maximum element and a minimum element.
\item $\Dc(P) = \Uc(P) = \{P\}$.
\end{enumerate}
\end{prop}

\begin{proof}
We will prove first that $(a) \Rightarrow (b)$. Let $C\in\Dc(P)\cap \Uc(P)$. Then $C$ is a non-empty down-set and up-set of $P$. Since $P$ is connected, $C=P$. Since $C\in\Dc(P)$, there exists a non-empty subset $A\subseteq \mxl(P)$ such that $C$ is a connected component of $\st{A}$. Hence $\st{A}=P$ and thus $\#A=1$. Then, if $a$ is the unique element of $A$, we obtain that $P=P_{\leq a}$. Hence, $P$ has a maximum element. In a similar way, it follows that $P$ has a minimum element.

The implication $(b) \Rightarrow (c)$ follows immediately from the definition of $\Dc(P)$ and $\Uc(P)$, while the implication $(c)\Rightarrow (a)$ is trivial.
\end{proof}

Observe that the previous result does not hold if the poset $P$ is not connected. Indeed, if $P$ is the antichain poset on $\{0,1\}$, then $\Dc(P)=\Uc(P)=\{\{0\},\{1\}\}$.

\medskip

The following important definition relies heavily on Proposition \ref{prop_connected_subset}.

\begin{definition} \label{def_Df_Uf_Cf}
Let $P$ and $Q$ be posets and let $f\colon P\to Q$ be an order-preserving map. 
\begin{itemize}
\item Suppose that $\mxl(Q)$ is a cutset of $Q$. We define 
\begin{displaymath}
\Dc(f)\colon \Dc(P)\to \Dc(Q) \qquad \textnormal{by} \qquad \Dc(f)(C)= \min \{D\in\Dc(Q) \mid f(C)\subseteq D \}.
\end{displaymath}
\item Suppose that $\mnl(Q)$ is a cutset of $Q$. We define 
\begin{displaymath}
\Uc(f)\colon \Uc(P)\to \Uc(Q) \qquad \textnormal{by} \qquad \Uc(f)(C)= \max \{D\in\Uc(Q) \mid f(C)\subseteq D \}.
\end{displaymath}
\item Suppose that both $\mxl(Q)$ and $\mnl(Q)$ are cutsets of $Q$ and that $\Dc(P)\cap\Uc(P)=\varnothing$. We define 
\begin{displaymath}
\Cc(f)\colon \Cc(P)\to \Cc(Q) \qquad \textnormal{by} \qquad \Cc(f)(C)= 
\begin{cases}
\Dc(f)(C) & \textup{if } C\in\Dc(P), \\
\Uc(f)(C) & \textup{if } C\in\Uc(P).
\end{cases}
\end{displaymath}
\end{itemize}
\end{definition}

We will prove now that $\Dc(f)$ is well defined. Let $C\in\Dc(P)$. Then there exists $a\in \mxl(P)$ such that $C\subseteq P_{\leq a}$. Since $\mxl(Q)$ is a cutset of $Q$, there exists $b\in\mxl(Q)$ such that $f(a)\leq b$. Then $f(C)\subseteq Q_{\leq f(a)}\subseteq Q_{\leq b}$. Hence, $\{D\in\Dc(Q) \mid f(C)\subseteq D \}$ is non-empty and thus it has a minimum element by Proposition \ref{prop_connected_subset}. In a similar way it can be proved that $\Uc(f)$ is well defined and thus $\Cc(f)$ is also well defined.

It is not difficult to verify that the maps $\Dc(f)$, $\Uc(f)$ and $\Cc(f)$ are order-preserving.

\medskip

The following example shows that the assignment $f\mapsto \Dc(f)$ of the previous definition is not functorial, and that a similar situation occurs with $\Uc$ and $\Cc$.

\begin{ex}
Consider the poset $P$ and its associated poset $\Dc(P)$ which are given below.
\smallskip
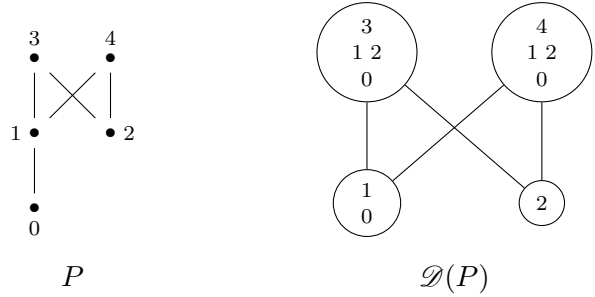
\begin{center}
\begin{tikzpicture}[baseline=-1.25cm]
\tikzstyle{every node}=[font=\scriptsize]
\draw (0,0) node(0){$\bullet$} node[below=1]{0}; 
\draw (0,1) node (1){$\bullet$} node[left=1]{1};
\draw (1,1) node (2){$\bullet$} node[right=1]{2};
\foreach \x in {3,4} \draw (\x-3,2) node(\x){$\bullet$} node[above=1]{\x}; 
  
\draw (0)--(1);
\foreach \x in {4,3} \draw (1)--(\x);
\foreach \x in {4,3} \draw (2)--(\x);
\draw (0.5,-0.9) node(P){\normalsize $P$};
\end{tikzpicture}
\hspace*{2cm}
\begin{tikzpicture}[x=2.3cm, y=2cm]
\path node (U3) at (0,1) [CC] {3\\1 2\\0}
 node (U4) at (1,1) [CC] {4\\1 2\\0}
 node (U1) at (0,0) [CC] {1\\0}
 node (U2) at (1,0) [CC] {2};

\draw (U1)--(U4)--(U2)--(U3)--(U1);
\draw (0.5,-0.5) node(UP){$\Dc(P)$};
\end{tikzpicture}
\end{center}
Let $f\colon P\to P$ be the order-preserving map defined by
\begin{center}
\begin{tabular}{c | c c c c c}
$x$ & 0 & 1 & 2 & 3& 4\\[-2pt]
\hline
$f(x)$& 2 & 3 & 3 & 3& 3
\end{tabular}
\end{center}
Let $g\colon P\to P$ be the constant map with value 0. Then $fg$ is the constant map with value 2. It follows that, for all $C\in\Dc(P)$, $\Dc(g)(C)=\{0,1\}$ and $\Dc(fg)(C)=\{2\}$. Thus, for all $C\in\Dc(P)$, $\Dc(f)\Dc(g)(C)=\Dc(f)(\{0,1\})= P_{\leq 3}$. Hence, $\Dc(fg)\neq\Dc(f)\Dc(g)$. 

This implies that $\Cc(fg)\neq\Cc(f)\Cc(g)$, and the fact that $\Uc$ is not functorial can be deduced from this example considering opposite posets.
\end{ex}

\medskip 

We will now show how the induced maps of Definition \ref{def_Df_Uf_Cf} can be applied to give an alternative proof to the following result of H. H\"oft and M. H\"oft.

\begin{prop}[{\cite[Theorem 2]{hoft1976some}}]
Let $P$ be a poset such that $\mnl(P)$ is a finite cutset of $P$. Let $f\colon P\to P$ be an order-preserving map. Suppose that each non-empty subset of $\mnl(P)$ has a join. Then there exists $x\in P$ such that $f(x)\geq x$. 
\end{prop}

\begin{proof}
Note that for each non-empty subset $A$ of $\mnl(P)$ we have that $\st{A}$ has a minimum element since the subset $A$ has a join. In particular $\Uc(P)=\{\st{A} \mid A \in \mathcal{P}_{\neq\varnothing}(\mnl(P))\}$. It follows that the the poset $\Uc(P)$ has the fixed point property since it is finite and has a maximum element (which is $\st{\mnl(P)}$). Thus, the order-preserving map $\Uc(f)$ has a fixed point $C_0$. Hence, $f(C_0)\subseteq \Uc(f)(C_0)=C_0$. Since $C_0\in \Uc(P)$ there exists a non-empty subset $A$ of $\mnl(P)$ such that $C_0=\st{A}$. Let $a=\min(\st{A})$. Then, $f(a)\in f(C_0)\subseteq C_0 = \st{A}$, and hence $f(a)\geq a$.
\end{proof}

The following proposition shows a simple but interesting relationship between the induced maps of Definition \ref{def_Df_Uf_Cf} and the fixed point property.

\begin{prop} \label{prop_fpp}
Let $P$ be a poset and let $f\colon P\to P$ be an order-preserving map.
\begin{enumerate}[(a)]
\item Suppose that $\mxl(P)$ is a cutset of $P$. If $C\in \Dc(P)$ is a fixed point of the map $\Dc(f)$ and the subposet $C$ has the fixed point property then the map $f$ has a fixed point in $C$.
\item Suppose that both $\mxl(P)$ and $\mnl(P)$ are cutsets of $P$ and that $\Dc(P)\cap\Uc(P)=\varnothing$. If $C\in \Cc(P)$ is a fixed point of the map $\Cc(f)$ and the subposet $C$ has the fixed point property then the map $f$ has a fixed point in $C$.
\end{enumerate}
\end{prop}

\begin{proof} 
Under the hypotheses of item $(a)$ we have that $f(C)\subseteq \Dc(f)(C)= C$ and since the subposet $C$ has the fixed point property we obtain that the map $f$ has a fixed point in $C$. Item ($b$) can be proved in a similar way.
\end{proof}

The following easy example shows that the previous result may not hold without the hypothesis that the subposet $C$ has the fixed point property.

\begin{ex} \label{ex_2}
Consider the poset $P$ and its associated poset $\Dc(P)$ which are given below.
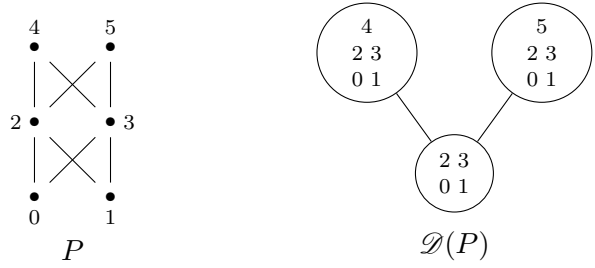
\begin{center}
\begin{tikzpicture}
\tikzstyle{every node}=[font=\scriptsize]
\foreach \x in {0,1} \draw (\x,0) node(\x){$\bullet$} node[below=1]{\x}; 
\draw (0,1) node (2){$\bullet$} node[left=1]{2};
\draw (1,1) node (3){$\bullet$} node[right=1]{3};
\foreach \x in {4,5} \draw (\x-4,2) node(\x){$\bullet$} node[above=1]{\x}; 
  
\foreach \x in {2,3} \draw (0)--(\x);
\foreach \x in {2,3} \draw (1)--(\x);
\foreach \x in {4,5} \draw (2)--(\x);
\foreach \x in {4,5} \draw (3)--(\x);
\draw (0.5,-0.7) node(P){\normalsize $P$};
\end{tikzpicture}
\hspace*{2cm}
\begin{tikzpicture}[x=2.3cm, y=1.6cm]
\path node (U4) at (0,1) [CC] {4\\2 3\\0 1}
 node (U5) at (1,1) [CC] {5\\2 3\\ 0 1}
 node (U45) at (0.5,0) [CC] {2 3\\0 1};
 
\draw [-] (U45) to (U4); 
\draw [-] (U45) to (U5);
\draw (0.5,-0.6) node(UP){$\Dc(P)$};
\end{tikzpicture}
\end{center}
Note that the poset $\Dc(P)$ has the fixed point property but the poset $P$ does not have the fixed point property.
\end{ex}

The following result follows immediately from Proposition \ref{prop_fpp}. Its proof wil be omitted.

\begin{coro} \label{coro_fpp}
Let $P$ be a poset. 
\begin{enumerate}[(a)]
\item Suppose that $\mxl(P)$ is a cutset of $P$. If the poset $\Dc(P)$ has the fixed point property and every $C\in\Dc(P)$ has the fixed point property then $P$ has the fixed point property.
\item Suppose that both $\mxl(P)$ and $\mnl(P)$ are cutsets of $P$ and that $\Dc(P)\cap\Uc(P)=\varnothing$. If the poset $\Cc(P)$ has the fixed point property and every $C\in\Cc(P)$ has the fixed point property then $P$ has the fixed point property.
\end{enumerate}
\end{coro}

Corollary \ref{coro_fpp} generalizes \cite[Corollary 5.3]{baclawski1979fixed} (which is itself a generalization of \cite[Theorem 2]{duffus1980retracts}). Indeed, under the assumptions of \cite[Corollary 5.3]{baclawski1979fixed}, for every non-empty subset $A\subseteq \mxl(P^\op)$ the subset $\st{A}$ is non-empty and connected (since it has the fixed point property) and thus $\Dc(P^{\op})=\{\st{A} \mid A \in \mathcal{P}_{\neq\varnothing}(\mxl(P^\op))\}$. In particular, the poset $\Dc(P^{\op})$ is finite and has a minimum element (which is $\st{\mxl(P^\op)}$), and thus, it has the fixed point property. It is worth mentioning that in Corollary \ref{coro_fpp} we require neither that the subsets $\st{A}$ are connected for every non-empty subset $A\subseteq \mxl(P^\op)$ nor that the subset of maximal elements of $P$ is finite.

The following is a kind of converse of Proposition \ref{prop_fpp}.

\begin{prop} \label{prop_fpp_converse}
Let $P$ be a poset such that $\mxl(P)$ is a finite cutset of $P$. Let $f\colon P\to P$ be an order-preserving map. If $f$ has a fixed point, then the map $\Dc(f)$ has a fixed point.
\end{prop}

\begin{proof}
We will prove first that the poset $\Dc(P)$ does not have infinite ascending chains. Let $C_1,C_2\in\Dc(P)$ such that $C_1< C_2$. Then $C_1\subsetneq C_2$ and thus $\I_{\mxl(P)}(C_1)\supseteq \I_{\mxl(P)}(C_2)$. If $\I_{\mxl(P)}(C_1)= \I_{\mxl(P)}(C_2)$ then, by Proposition \ref{prop_connected_subset}, both $C_1$ and $C_2$ are connected components of $\st{\I_{\mxl(P)}(C_1)}$, and since $C_1\subseteq C_2$ we obtain that $C_1=C_2$, which entails a contradiction. Thus, $\I_{\mxl(P)}(C_1)\supsetneq \I_{\mxl(P)}(C_2)$. Since $\mxl(P)$ is finite we obtain that $\Dc(P)$ does not have infinite ascending chains.

Let $x\in P$ be such that $f(x) = x$. Since $\mxl(P)$ is a cutset of $P$, there exists $a\in\mxl(X)$ such that $x\in P_{\leq a}$. By Proposition \ref{prop_connected_subset}, the set $\{C\in\Dc(P) \mid x\in C\}$ has a minimum element $C_0$. It follows that $x = f(x)\in f(C_0)\subseteq \Dc(f)(C_0)$. Thus, $C_0\subseteq \Dc(f)(C_0)$. Hence, $\Dc(f)$ has a fixed point by the Abian-Brown theorem (\cite[Theorem 2]{abian1961theorem}).
\end{proof}

\section{Examples}

Finally, we will give several examples in which we apply Proposition \ref{prop_fpp} (or Corollary \ref{coro_fpp}) to prove that certain posets have the fixed point property. These examples will show how Proposition \ref{prop_fpp} can be combined with other arguments and tools to prove that a given poset has the fixed point property. We will need the following lemma.

\begin{lemma}
\label{lemma_examples}
 Let $P_1$ and $P_2$ be the posets given by the following Hasse diagrams.
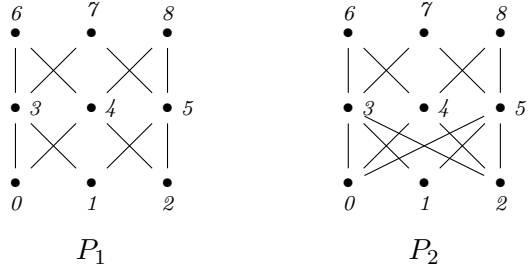
\begin{center}
 \begin{tikzpicture}
 \tikzstyle{every node}=[font=\scriptsize]
  
  \foreach \x in {0,1,2} \draw (\x,0) node(\x){$\bullet$} node[below=1]{\x}; 
  \foreach \x in {3,4,5} \draw (\x-3,1) node(\x){$\bullet$} node[right=1]{\x}; 
  \foreach \x in {6,7,8} \draw (\x-6,2) node(\x){$\bullet$} node[above=1]{\x};
  \draw(1,-0.9) node {\normalsize $P_1$};
  
  \foreach \x in {3,4} \draw (0)--(\x);
  \foreach \x in {3,5} \draw (1)--(\x);
  \foreach \x in {4,5} \draw (2)--(\x);
 
  \foreach \x in {6,7} \draw (3)--(\x);
  \foreach \x in {6,8} \draw (4)--(\x);
  \foreach \x in {7,8} \draw (5)--(\x);
 \end{tikzpicture}
\qquad \qquad
 \begin{tikzpicture}
 \tikzstyle{every node}=[font=\scriptsize]
  \foreach \x in {0,1,2} \draw (\x,0) node(\x){$\bullet$} node[below=1]{\x}; 
  \foreach \x in {3,4,5} \draw (\x-3,1) node(\x){$\bullet$} node[right=1]{\x}; 
  \foreach \x in {6,7,8} \draw (\x-6,2) node(\x){$\bullet$} node[above=1]{\x}; 
  \draw(1,-0.9) node {\normalsize $P_2$};
  
   \foreach \x in {3,4,5} \draw (0)--(\x);
   \foreach \x in {3,5} \draw (1)--(\x);
   \foreach \x in {3,4,5} \draw (2)--(\x);
 
   \foreach \x in {6,7} \draw (3)--(\x);
   \foreach \x in {6,8} \draw (4)--(\x);
   \foreach \x in {7,8} \draw (5)--(\x);
 \end{tikzpicture}
\end{center}
Let $P$ be either $P_1$ or $P_2$. If $f\colon P\to P$ is an order-preserving map without fixed points, then $f(\{3,4,5\})=\{3,4,5\}$.
\end{lemma}

\begin{proof}
We will prove first that $f(\{3,4,5\})\subseteq P-\mnl(P)$. Let $x\in \{3,4,5\}$ and suppose that $f(x)\in \mnl(P)$. Then $f(P_{\leq x})=\{f(x)\}$ and hence there exist $a,b\in \mnl(P)$ such that $a\neq b$ and $f(a)=f(b)=f(x)$. Note that $f(x)\nleq x$ since $f$ does not have fixed points and thus if $P=P_2$ then $f(x)=1$ and $x=4$. Since $a\neq b$ we obtain that $P-\mnl(P)\subseteq P_{\geq a} \cup P_{\geq b}$ and thus $f(P_{>f(x)})\subseteq f(P_{\geq a} \cup P_{\geq b}) \subseteq P_{\geq f(x)}$. If there exists $z\in P_{>f(x)}$ such that $f(z)=f(x)$ then $f(z)<z$ and hence $f$ has a fixed point, contradicting the hypothesis on $f$. Thus, $f(P_{>f(x)})\subseteq P_{>f(x)}$ and since $P_{>f(x)}$ is dismantlable by irreducibles we obtain that $f$ has a fixed point, which entails a contradiction. Therefore, $f(\{3,4,5\})\subseteq P-\mnl(P)$. 

It follows that $f(P-\mnl(P))\subseteq P-\mnl(P)$. And since $f$ does not have fixed points, $f$ is bijective in this subposet by Remark \ref{rem_crown}. In particular, $f(\{3,4,5\})=\{3,4,5\}$.
\end{proof}

In the following examples we will prove that the posets $P^{3323}$, $P^{353}_1$ and $P^{353}_2$ of \cite[Fig.1]{schroder1993fpp11} have the fixed point property. Their Hasse diagrams are given below.

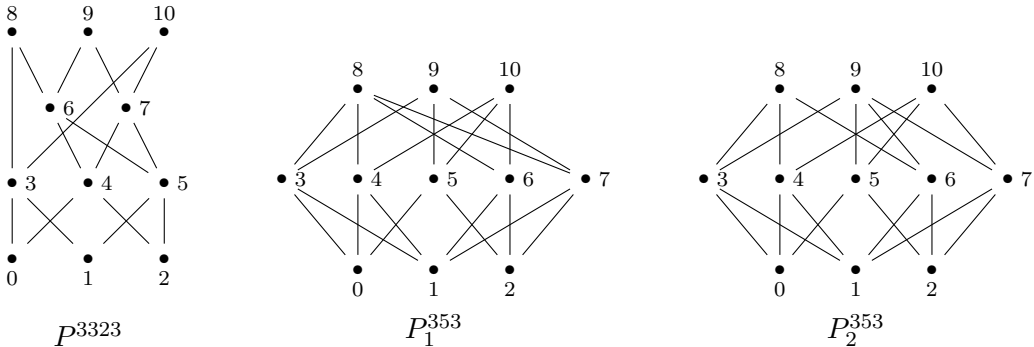
\begin{center}
\begin{tikzpicture}
\tikzstyle{every node}=[font=\scriptsize]
\foreach \x in {0,1,2} \draw (\x,0) node(\x){$\bullet$} node[below=1]{\x}; 
\foreach \x in {3,4,5} \draw (\x-3,1) node(\x){$\bullet$} node[right=1]{\x}; 
\foreach \x in {6,7} \draw (\x-5.5,2) node(\x){$\bullet$} node[right=1]{\x}; 
\foreach \x in {8,9,10} \draw (\x-8,3) node(\x){$\bullet$} node[above=1]{\x}; 
  
\foreach \x in {3,4} \draw (0)--(\x);
\foreach \x in {3,5} \draw (1)--(\x);
\foreach \x in {4,5} \draw (2)--(\x);
\foreach \x in {8,10} \draw (3)--(\x);
\foreach \x in {6,7} \draw (4)--(\x);
\foreach \x in {6,7} \draw (5)--(\x);
\foreach \x in {8,9} \draw (6)--(\x);
\foreach \x in {9,10} \draw (7)--(\x);

\draw (1,-1) node(L){\normalsize $P^{3323}$};
\end{tikzpicture}
\qquad
\begin{tikzpicture}
\tikzstyle{every node}=[font=\scriptsize]
\foreach \x in {0,1,2} \draw (\x,-0.2) node(\x){$\bullet$} node[below=1]{\x}; 
\foreach \x in {3,4,5,6,7} \draw (\x-4,1) node(\x){$\bullet$} node[right=1]{\x}; 
\foreach \x in {8,9,10} \draw (\x-8,2.2) node(\x){$\bullet$} node[above=1]{\x}; 
  
\foreach \x in {3,4,5} \draw (0)--(\x);
\foreach \x in {3,4,6,7} \draw (1)--(\x);
\foreach \x in {5,6,7} \draw (2)--(\x);

\foreach \x in {3,4,6,7} \draw (8)--(\x);
\foreach \x in {3,5,7} \draw (9)--(\x);
\foreach \x in {4,5,6} \draw (10)--(\x);

\draw (1,-1) node(L){\normalsize $P^{353}_1$};
\end{tikzpicture}
\qquad
\begin{tikzpicture}
\tikzstyle{every node}=[font=\scriptsize]
\foreach \x in {0,1,2} \draw (\x,-0.2) node(\x){$\bullet$} node[below=1]{\x}; 
\foreach \x in {3,4,5,6,7} \draw (\x-4,1) node(\x){$\bullet$} node[right=1]{\x}; 
\foreach \x in {8,9,10} \draw (\x-8,2.2) node(\x){$\bullet$} node[above=1]{\x}; 
  
\foreach \x in {3,4,5} \draw (0)--(\x);
\foreach \x in {3,4,6,7} \draw (1)--(\x);
\foreach \x in {5,6,7} \draw (2)--(\x);

\foreach \x in {3,4,6} \draw (8)--(\x);
\foreach \x in {3,5,6,7} \draw (9)--(\x);
\foreach \x in {4,5,7} \draw (10)--(\x);

\draw (1,-1) node(L){\normalsize $P^{353}_2$};
\end{tikzpicture}
\end{center}

\begin{ex}
\label{ex_1_fpp}
In this example we will prove that the poset $P^{3323}$ has the fixed point property. The Hasse diagrams of $\Dc(P^{3323})$ and $\Uc(P^{3323})$ are shown below.  The shaded elements of these diagrams are subposets of $P^{3323}$ which have the fixed point property (since they are dismantlable by irreducibles).

\smallskip

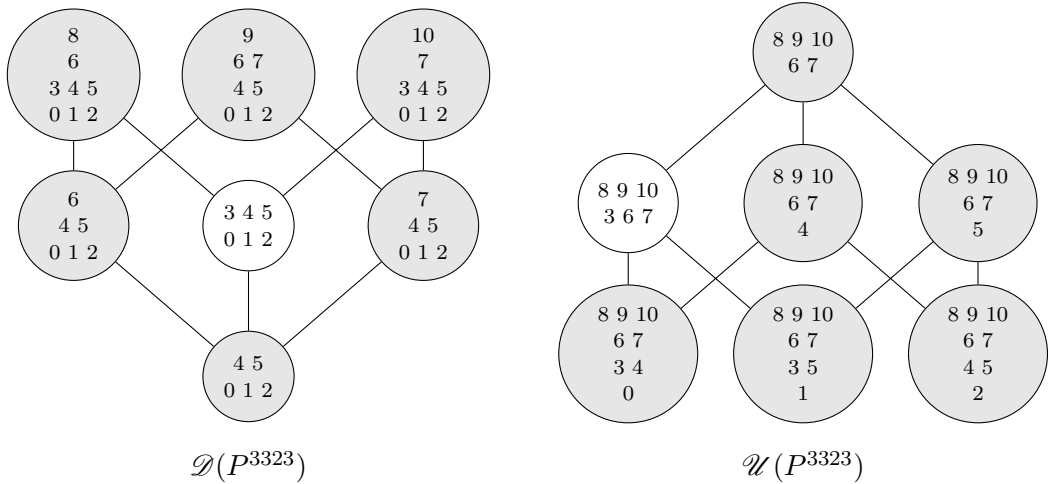
\begin{center}
\begin{tikzpicture}[x=2.3cm, y=2cm]
\path node[CCC] (U8) at (0,6) { 8\\ 6\\ 3 4 5\\ 0 1 2}
node[CCC] (U9) at (1,6) { 9\\ 6 7\\ 4 5\\ 0 1 2}
node (U10) at (2,6) [CCC] { 10\\ 7\\ 3 4 5\\ 0 1 2}

node (U89) at (0,5) [CCC] { 6\\ 4 5\\ 0 1 2} 
node (U810) at (1,5) [CC] { 3 4 5\\ 0 1 2} 
node (U910) at (2,5) [CCC] { 7\\ 4 5\\ 0 1 2}

node (U8910) at (1,4) [CCC] { 4 5\\ 0 1 2};

\draw [-] (U8) to (U810) to (U10) to (U910) to (U9) to (U89) to (U8);
\draw [-] (U8910) to (U89);
\draw [-] (U8910) to (U810);
\draw [-] (U8910) to (U910);

\draw (1,3.4) node(L){\normalsize $\Dc(P^{3323})$};
\end{tikzpicture}
\qquad
\begin{tikzpicture}[x=2.3cm, y=2cm]
\path node (F012) at (1,3) [CCC] { 8 9 10\\ 6 7}

node (F01) at (0,2) [CC] { 8 9 10\\ 3 6 7} 
node (F02) at (1,2) [CCC] { 8 9 10\\ 6 7\\ 4} 
node (F12) at (2,2) [CCC] { 8 9 10\\ 6 7\\ 5}

node (F0) at (0,1) [CCC] { 8 9 10\\ 6 7\\ 3 4\\ 0} 
node (F1) at (1,1) [CCC] { 8 9 10\\ 6 7\\ 3 5\\ 1} 
node (F2) at (2,1) [CCC] { 8 9 10\\ 6 7\\ 4 5\\ 2};

\draw[-] (F012) to (F01);
\draw[-] (F012) to (F02);
\draw[-] (F012) to (F12);
\draw[-] (F0) to (F02) to (F2) to (F12) to (F1) to (F01) to (F0);

\draw (1,0.25) node(L){\normalsize $\Uc(P^{3323})$};
\end{tikzpicture}
\end{center}

Let $f\colon P^{3323}\to P^{3323}$ be an order-preserving map. Let $C_1=\{0,1,2,3,4,5\}\in \Dc(P^{3323})$ and let $C_2=\{3,6,7,8,9,10\}\in \Uc(P^{3323})$. Note that the finite posets $\Dc(P^{3323})$ and $\Uc(P^{3323})$ have the fixed point property.  Let $D_1 \in \Dc(P^{3323})$ and $D_2\in \Uc(P^{3323})$ be fixed points of the maps $\Dc(f)$ and $\Uc(f)$ respectively. If $D_1\neq C_1$ then $D_1$ has the fixed point property and hence, applying Proposition \ref{prop_fpp}, we obtain that $f$ has a fixed point in $D_1$. Analogously, if $D_2\neq C_2$ then $f$ has a fixed point in $D_2$. Thus, we may assume that $D_1=C_1$ and $D_2=C_2$. Hence, $f(C_1)\subseteq \Dc(f)(C_1)=C_1$ and $f(C_2)\subseteq \Uc(f)(C_2)=C_2$. Then, $f(C_1\cap C_2)\subseteq C_1\cap C_2$ and hence $f(3)=3$. Therefore, the poset $P^{3323}$ has the fixed point property.
\end{ex}

\begin{ex}
In this example we will prove that the poset $P^{353}_1$ has the fixed point property. The Hasse diagram of $\Cc(P^{353}_1)$ is shown in Figure \ref{fig_ex2}.

\begin{figure}
\begin{subfigure}[c]{0.48\textwidth}
\centering
\begin{tikzpicture}[x=2.3cm, y=2cm]
\path node[CCC] (U8) at (0,6)  { 8\\ 3 4 6 7\\ 0 1 2}
node (U9) at (1,6) [CCC] { 9\\ 3 5 7\\ 0 1 2}
node (U10) at (2,6) [CCC] { 10\\ 4 5 6\\ 0 1 2}

node (U89) at (0,5) [CCC] { 3 7\\ 0 1 2} 
node (U810) at (1,5) [CCC] { 4 6\\ 0 1 2} 
node (U910a) at (2,5) [CCC] { 5\\ 0 2}

node (U8910a) at (0.2,4) [CCC] { 0}
node (U8910b) at (1,4) [CCC] { 1}
node (U8910c) at (1.8,4) [CCC] { 2};

\draw [-] (U8) to (U810) to (U10) to (U910a) to (U9) to (U89) to (U8);

\foreach \x in {U89, U810, U910a} \draw (U8910a)--(\x);
\foreach \x in {U89, U810} \draw (U8910b)--(\x);
\foreach \x in {U89, U910a, U810} \draw (U8910c)--(\x);

\path node (F012a) at (0.2,3) [CCC] { 8}
node (F012b) at (1,3) [CCC] { 9}
node (F012c) at (1.8,3) [CCC] { 10}

node (F01) at (0,2) [CCC] { 8 9 10\\ 3 4} 
node (F02a) at (1,2) [CCC] { 9 10\\ 5} 
node (F12) at (2,2) [CCC] { 8 9 10\\ 6 7}

node (F0) at (0,1) [CCC] { 8 9 10\\ 3 4 5\\ 0} 
node (F1) at (1,1) [CCC] { 8 9 10\\ 3 4 6 7\\ 1} 
node (F2) at (2,1) [CCC] { 8 9 10\\ 5 6 7\\ 2};

\foreach \x in {F01, F12} \draw (F012a)--(\x);
\foreach \x in {F01, F02a, F12} \draw (F012b)--(\x);
\foreach \x in {F01, F02a, F12} \draw (F012c)--(\x);
\draw[-] (F0) to (F02a) to (F2) to (F12) to (F1) to (F01) to (F0);

\draw[-, red] (U810) to (F01) to (U89) to (F12) to (U810);
\draw[-, red] (F02a) to (U910a);
\draw[-, blue, bend left] (U8) to (F012a);
\draw[-, blue, bend left] (U9) to (F012b);
\draw[-, blue, bend left] (U10) to (F012c);

\draw[-, blue, bend left] (F0) to (U8910a);
\draw[-, blue, bend left] (F1) to (U8910b);
\draw[-, blue, bend left] (F2) to (U8910c);
\end{tikzpicture}
\caption{$\Cc(P^{353}_1)$}
\label{fig_ex2}
\end{subfigure}
\begin{subfigure}[c]{0.48\textwidth}
\centering
\begin{tikzpicture}[x=2.3cm, y=2cm]
\path node (U8) at (0,6) [CCC] { 8\\ 3 4 6\\ 0 1 2}
node (U9) at (1,6) [CCC] { 9 \\ 3 5 6 7\\ 0 1 2}
node (U10) at (2,6) [CCC] { 10\\ 4 5 7\\ 0 1 2}

node (U89) at (0,5) [CCC, label=left:$A$] { 3 6\\ 0 1 2} 
node (U810a) at (1,5) [CCC, label=left:$B$] { 4\\ 0 1} 
node (U910) at (2,5) [CCC, label=left:$C$] { 5 7\\ 0 1 2}

node (U8910a) at (0.2,4) [CCC] { 0}
node (U8910b) at (1,4) [CCC] { 1}
node (U8910c) at (1.8,4) [CCC] { 2};

\draw [-] (U8) to (U810a) to (U10) to (U910) to (U9) to (U89) to (U8);

\foreach \x in {U89, U810a, U910} \draw (U8910a)--(\x);
\foreach \x in {U89, U810a, U910} \draw (U8910b)--(\x);
\foreach \x in {U89, U910} \draw (U8910c)--(\x);

\path node (F012a) at (0.2,3) [CCC] {8}
node (F012b) at (1,3) [CCC] {9}
node (F012c) at (1.8,3) [CCC] {10}

node (F01) at (0,2) [shape=circle, draw, fill=black!10, label=right:$A'$, align=center, font=\scriptsize] {8 9 10\\3 4} 
node (F02a) at (1,2) [shape=circle, draw, fill=black!10, label=right:$B'$, align=center, font=\scriptsize] {9 10\\5} 
node (F12) at (2,2) [shape=circle, draw, fill=black!10, label=right:$C'$,  align=center, font=\scriptsize] {8 9 10\\6 7}

node (F0) at (0,1) [CCC] {8 9 10\\3 4 5\\0} 
node (F1) at (1,1) [CCC] {8 9 10\\3 4 6 7\\1} 
node (F2) at (2,1) [CCC] {8 9 10\\5 6 7\\2};

\foreach \x in {F01, F12} \draw (F012a)--(\x);
\foreach \x in {F01, F02a, F12} \draw (F012b)--(\x);
\foreach \x in {F01, F02a, F12} \draw (F012c)--(\x);
\draw[-] (F0) to (F02a) to (F2) to (F12) to (F1) to (F01) to (F0);

\draw[-, red] (U810a) to (F01) to (U89) to (F12) to (U910) to (F02a);
\draw[-, blue, bend left] (U8) to (F012a);
\draw[-, blue, bend left] (U9) to (F012b);
\draw[-, blue, bend left] (U10) to (F012c);

\draw[-, blue, bend left] (F0) to (U8910a);
\draw[-, blue, bend left] (F1) to (U8910b);
\draw[-, blue, bend left] (F2) to (U8910c);
\end{tikzpicture}
\caption{$\Cc(P^{353}_2)$}
\label{fig_ex3}
\end{subfigure}
\caption{}
\end{figure}

Note that the minimal elements of $\Dc(P^{353}_1)$ and the maximal elements of $\Uc(P^{353}_1)$ are irreducible points of $\Cc(P^{353}_1)$. The poset that is obtained by removing those irreducible points is the poset $Q_1$ given below. Now observe that the element labeled $a$ is an irreducible point of $Q_1$, and by removing it we obtain the poset $Q_2$ given below, which is isomorphic to $P^{3323}$.

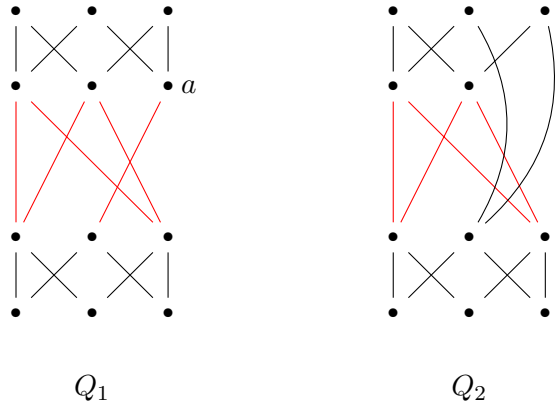
\begin{center}
\begin{tikzpicture}
 \tikzstyle{every node}=[font=\scriptsize]
  \foreach \x in {0,1,2} \draw (\x,0) node(\x){$\bullet$}; 
  \foreach \x in {3,4,5} \draw (\x-3,1) node(\x){$\bullet$}; 

  \foreach \x in {9,10} \draw (\x-9,3) node(\x){$\bullet$};
  \draw (2,3) node(11){$\bullet$} node[right=1]{{\normalsize $a$}};
  \foreach \x in {12,13,14} \draw (\x-12,4) node(\x){$\bullet$}; 

  \foreach \x in {3,4} \draw (0)--(\x);
  \foreach \x in {3,5} \draw (1)--(\x);
  \foreach \x in {4,5} \draw (2)--(\x);

  \foreach \x in {12,13} \draw (9)--(\x);
  \foreach \x in {12,14} \draw (10)--(\x);
  \foreach \x in {13,14} \draw (11)--(\x);

  \draw[-, red] (3) to (9) to (5) to (10) to (3);
  \draw[-, red] (4) to (11);
  
  \draw (1,-1) node(L){\normalsize $Q_1$};
\end{tikzpicture}
\hspace*{2cm}
\begin{tikzpicture}
  \tikzstyle{every node}=[font=\scriptsize]
  \foreach \x in {0,1,2} \draw (\x,0) node(\x){$\bullet$}; 
  \foreach \x in {3,4,5} \draw (\x-3,1) node(\x){$\bullet$}; 

  \foreach \x in {9,10} \draw (\x-9,3) node(\x){$\bullet$};
  \foreach \x in {12,13,14} \draw (\x-12,4) node(\x){$\bullet$}; 

  \foreach \x in {3,4} \draw (0)--(\x);
  \foreach \x in {3,5} \draw (1)--(\x);
  \foreach \x in {4,5} \draw (2)--(\x);

  \foreach \x in {12,13} \draw (9)--(\x);
  \foreach \x in {12,14} \draw (10)--(\x);

  \draw[-, red] (3) to (9) to (5) to (10) to (3);
  \draw[-,bend right] (4) to (13);
  \draw[-,bend right] (4) to (14);
  
  \draw (1,-1) node(L){\normalsize $Q_2$};
\end{tikzpicture}
\end{center}

Since the poset $P^{3323}$ has the fixed point property (by Example \ref{ex_1_fpp}) we obtain that the poset $\Cc(P^{353}_1)$ has the fixed point property. And since all the elements of $\Cc(P^{353}_1)$ have the fixed point property (because all of them are dismantlable by irreducibles), we deduce from Corollary \ref{coro_fpp} that $P^{353}_1$ has the fixed point property.
\end{ex}

\begin{ex}
In this example we will prove that the poset $P^{353}_2$ has the fixed point property. The Hasse diagram of $\Cc(P^{353}_2)$ is shown in Figure \ref{fig_ex3}. Observe that all the elements of $\Cc(P^{353}_2)$ are dismantlable by irreducibles and thus they all have the fixed point property. Let $A$, $B$, $C$, $A'$, $B'$ and $C'$ be the sets indicated in Figure \ref{fig_ex3}.

Let $f\colon P^{353}_2\to P^{353}_2$ be an order-preserving map. If the map $\Dc(f)$ has a fixed point applying Proposition \ref{prop_fpp} we obtain that $f$ has a fixed point. Thus, we may assume that $\Dc(f)$ does not have fixed points. Hence, by Lemma \ref{lemma_examples}, $\Dc(f)(\{A,B,C\})\subseteq \{A,B,C\}$. In a similar way, we may assume that $\Uc(f)$ does not have fixed points and thus $\Uc(f)(\{A',B',C'\})\subseteq \{A',B',C'\}$ by Lemma \ref{lemma_examples}. Let $M = \{A,B,C,A',B',C'\}\subseteq \Cc(P^{353}_2)$. Since $M$ is dismantlable by irreducibles and $\Cc(f)(M)\subseteq M$, the map $\Cc(f)$ has a fixed point. Thus, applying Proposition \ref{prop_fpp}, we obtain that $f$ has a fixed point. Therefore, the poset $P^{353}_2$ has the fixed point property.
\end{ex}

\begin{ex}
Let $n,k\in\N$ such that $n\geq 4$ and $2\leq k \leq n-1$. Let $P^{n,k}$ be the poset whose Hasse diagram is
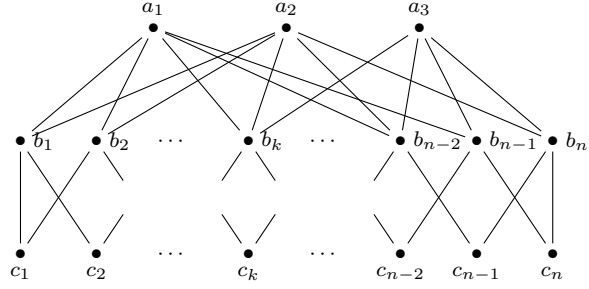
\begin{center}
\begin{tikzpicture}[y=1.5cm]
\tikzstyle{every node}=[font=\scriptsize]
\foreach \x in {1,2} \draw (\x,0) node(c\x)[inner sep=2pt]{$\bullet$} node[below=1]{$c_\x$}; 
\foreach \x in {3,5} \draw (\x,0) node(c\x)[inner sep=2pt]{$\cdots$};
\foreach \x in {4} \draw (\x,0) node(c\x)[inner sep=2pt]{$\bullet$} node[below=1]{$c_k$}; 
\foreach \x in {6} \draw (\x,0) node(c\x)[inner sep=2pt]{$\bullet$} node[below=1]{$c_{n-2}$}; 
\foreach \x in {7} \draw (\x,0) node(c\x)[inner sep=2pt]{$\bullet$} node[below=1]{$c_{n-1}$}; 
\foreach \x in {8} \draw (\x,0) node(c\x)[inner sep=2pt]{$\bullet$} node[below=1]{$c_n$}; 

\foreach \x in {1,2} \draw (\x,1) node(b\x)[inner sep=2pt]{$\bullet$} node[right=1]{$b_\x$}; 
\foreach \x in {3,5} \draw (\x,1) node(b\x)[inner sep=2pt]{$\cdots$};
\foreach \x in {4} \draw (\x,1) node(b\x)[inner sep=2pt]{$\bullet$} node[right=1]{$b_k$}; 
\foreach \x in {6} \draw (\x,1) node(b\x)[inner sep=2pt]{$\bullet$} node[right=1]{$b_{n-2}$}; 
\foreach \x in {7} \draw (\x,1) node(b\x)[inner sep=2pt]{$\bullet$} node[right=1]{$b_{n-1}$}; 
\foreach \x in {8} \draw (\x,1) node(b\x)[inner sep=2pt]{$\bullet$} node[right=1]{$b_n$}; 

\foreach \x in {1,2,3} \draw (1.75*\x+1,2) node(a\x)[inner sep=2pt]{$\bullet$} node[above=1]{$a_\x$}; 
  
\foreach \x in {1,2,4,6,7} \draw (a1)--(b\x);
\foreach \x in {1,2,4,6,8} \draw (a2)--(b\x);
\foreach \x in {4,6,7,8} \draw (a3)--(b\x);
\foreach \x in {1,2} \draw (b1)--(c\x);
\foreach \x in {1} \draw (b2)--(c\x);
\draw[shorten >=1cm,shorten <=0cm] (b2)--(c3);
\draw[shorten >=1cm,shorten <=0cm] (c2)--(b3);
\foreach \x in {3,5} \draw[shorten >=1cm,shorten <=0cm] (b4)--(c\x);
\foreach \x in {3,5} \draw[shorten >=1cm,shorten <=0cm] (c4)--(b\x);
\draw[shorten >=1cm,shorten <=0cm] (b6)--(c5);
\draw[shorten >=1cm,shorten <=0cm] (c6)--(b5);
\foreach \x in {7} \draw (b6)--(c\x);
\foreach \x in {6,8} \draw (b7)--(c\x);
\foreach \x in {7,8} \draw (b8)--(c\x);
\end{tikzpicture}
\end{center}
Note that 
\begin{align*}
P^{n,k}_{\leq a_1}&=\{a_1\}\cup \{b_j\mid 1\leq j\leq n-1\}\cup \{c_j\mid 1\leq j\leq n\},\\
P^{n,k}_{\leq a_2}&=\{a_2\}\cup \{b_j\mid 1\leq j\leq n \land j\neq n-1 \}\cup \{c_j\mid 1\leq j\leq n\},\\
P^{n,k}_{\leq a_3}&=\{a_3\}\cup \{b_j\mid k\leq j\leq n\}\cup \{c_j\mid k-1\leq j\leq n\},\\
P^{n,k}_{\leq b_1}&=\{b_1,c_1,c_2\},\\
P^{n,k}_{\leq b_j}&=\{b_j,c_{j-1},c_{j+1}\}, \textup{ for $2\leq j\leq n-1$},\\ 
P^{n,k}_{\leq b_n}&=\{b_n,c_{n-1},c_n\},\\
P^{n,k}_{\leq c_j}&=\{c_j\}, \textup{ for $1\leq j\leq n$}.
\end{align*}

Observe that the posets $P^{4,3}$ and $P^{4,2}$ are isomorphic to the posets $P^{443}_4$ and $P^{443}_5$ of \cite[Fig.1]{schroder1993fpp11}, respectively.

We will prove that the posets $P^{n,k}$ have the fixed point property for all $n,k$ as above. To this end we will compute $\Dc(P^{n,k})$. Let
\begin{align*}
A &=\{b_j\mid 1\leq j \leq n-2\} \cup \{c_j\mid 1\leq j \leq n-1\}, \\
B &=\{b_j \mid k\leq j\leq n-1 \ \land\ j\not\equiv n \modu{2} \} \cup \{c_j \mid k-1 \leq j \leq n \ \land\ j\equiv n \modu{2} \}, \\
C &=\{b_j \mid k\leq j\leq n \ \land\ j\equiv n \modu{2} \} \cup \{c_j \mid k-1\leq j\leq n-1 \ \land\ j\not\equiv n \modu{2} \}\cup\{c_n\}, \\
D &=\{b_j \mid k\leq j\leq n-3 \ \land\ j\not\equiv n \modu{2} \} \cup \{c_j \mid k-1\leq j \leq n-2 \ \land\ j\equiv n \modu{2} \}, \\
E &=\{b_j \mid k\leq j\leq n-2 \ \land\ j\equiv n \modu{2} \} \cup \{c_j \mid k-1\leq j\leq n-1 \ \land\ j\not\equiv n \modu{2} \}, \\
F &= \{c_n\}, \textnormal{ and } \\
S &= P^{n,k}-\mxl(P^{n,k}).
\end{align*}
It is not difficult to verify that the connected components of $\st{\{a_1,a_2\}}$ are $A$ and $F$, the connected components of $\st{\{a_1,a_3\}}$ are $B$ and $E$, the connected components of $\st{\{a_2,a_3\}}$ are $C$ and $D$ and that the connected components of $\st{\{a_1,a_2,a_3\}}$ are $D$, $E$ and $F$. Hence $\Dc(P^{n,k})$ is the finite poset given by the following Hasse diagram 
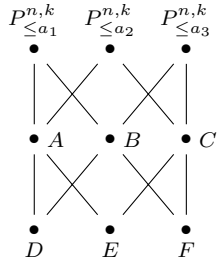
\begin{center}
\begin{tikzpicture}[x=1cm, y=1.2cm]
\tikzstyle{every node}=[font=\scriptsize]
\draw (0,2) node (Ua1){$\bullet$} node[above=1]{$P^{n,k}_{\leq a_1}$};
\draw (1,2) node (Ua2){$\bullet$} node[above=1]{$P^{n,k}_{\leq a_2}$};
\draw (2,2) node (Ua3){$\bullet$} node[above=1]{$P^{n,k}_{\leq a_3}$};
 
\draw (0,1) node (A){$\bullet$} node[right=1]{$A$};
\draw (1,1) node (B){$\bullet$} node[right=1]{$B$};
\draw (2,1) node (C){$\bullet$} node[right=1]{$C$};

\draw (0,0) node (D){$\bullet$} node[below=1]{$D$};
\draw (1,0) node (E){$\bullet$} node[below=1]{$E$};
\draw (2,0) node (F){$\bullet$} node[below=1]{$F$};
 
\draw [-] (Ua1) to (B) to (Ua3) to (C) to (Ua2) to (A) to (Ua1);
\draw [-] (A) to (E) to (C) to (F) to (B) to (D) to (A);
\end{tikzpicture}
\end{center}

Let $f\colon P^{n,k}\to P^{n,k}$ be an order-preserving map. Observe that all the elements of $\Dc(P^{n,k})$ have the fixed point property since they are dismantlable by irreducibles. Thus, if the map $\Dc(f)$ has a fixed point, applying Proposition \ref{prop_fpp} we obtain that $f$ has a fixed point. Hence, we may assume that $\Dc(f)$ does not have fixed points. Thus, by Lemma \ref{lemma_examples}, $\Dc(f)(\{A,B,C\})= \{A,B,C\}$. Since $S=A\cup B\cup C$, we have that 
\begin{displaymath}
f(S)= f(A)\cup f(B)\cup f(C) \subseteq \Dc(f)(A)\cup \Dc(f)(B)\cup \Dc(f)(C) = A\cup B\cup C = S.
\end{displaymath}
We will prove now that the restriction $f|\colon S\to S$ is not bijective. Note that 
\begin{align*}
\#A = 2n-3 > n-k +2 \geq \#B,
\end{align*}
where the first inequality holds since $n+k >5$. Thus, the sets $A$, $B$ and $C$ do not all have the same cardinality and since $\Dc(f)$ does not have fixed points and $\Dc(f)(\{A,B,C\})= \{A,B,C\}$ it follows that there exists $T\in \{A,B,C\}$ such that $\# \Dc(f)(T)< \# T$. Hence $\# f(T)< \# T$ and then $f|$ is not bijective. Thus, applying Remark \ref{rem_crown} we obtain that the map $f$ has a fixed point in $S$. Therefore, the poset $P^{n,k}$ has the fixed point property.
\end{ex}

\bibliographystyle{acm}
\bibliography{ref_crosscut}

\end{document}